\documentclass[12pt]{amsart}
\usepackage{latexsym,amsmath,amsfonts,amscd,amssymb}
\usepackage{stmaryrd}
\usepackage[english]{babel}

\usepackage[T1]{fontenc}
\usepackage{array}   
\usepackage[mathscr]{eucal} 
\usepackage{mathrsfs}
\usepackage{graphicx}
\usepackage{xypic}
\usepackage{calligra}

\usepackage[a4paper,top=3cm,bottom=2cm,left=3cm,right=3cm,marginparwidth=1.75cm]{geometry}

\usepackage{mathtools}

\usepackage[vcentermath]{youngtab}
\usepackage{young}

\usepackage[shortlabels]{enumitem}

\usepackage{amsthm}

\usepackage{relsize,exscale}

\usepackage{amsmath}
\usepackage{amssymb}
\usepackage{graphicx}
\usepackage[colorlinks=true, allcolors=blue]{hyperref}
\usepackage{cleveref}
\usepackage{pst-node}
\usepackage{tikz-cd}

\usepackage{graphicx}

\theoremstyle{plain}  
\newtheorem{theorem}{Theorem}[section]

\newtheorem*{theorem*}{Theorem}
\newtheorem{corollary}[theorem]{Corollary}
\newtheorem{lemma}[theorem]{Lemma}
\newtheorem{proposition}[theorem]{Proposition}

\theoremstyle{definition}
\newtheorem{definition}[theorem]{Definition}

\theoremstyle{remark}

\newtheorem{remark}[theorem]{Remark}
\newtheorem*{remark*}{Remark}

\newtheorem*{claim*}{Claim}

\DeclareMathOperator{\E}{E}
\DeclareMathOperator{\FF}{F}
\DeclareMathOperator{\PSL}{PSL}
\DeclareMathOperator{\rd}{red}
\DeclareMathOperator{\Sp}{Sp}
\DeclareMathOperator{\PSp}{PSp}
\DeclareMathOperator{\SO}{SO}

\DeclareMathOperator{\Spec}{Spec}
\DeclareMathOperator{\Hom}{Hom}

\DeclareMathOperator{\Id}{Id}
\DeclareMathOperator{\End}{End}

\DeclareMathOperator{\tr}{tr}

\DeclareMathOperator{\GL}{GL}
\DeclareMathOperator{\SL}{SL}
\DeclareMathOperator{\SU}{SU}
\DeclareMathOperator{\U}{U}
\DeclareMathOperator{\PU}{PU}

\DeclareMathOperator{\Log}{Log}
\DeclareMathOperator{\PGL}{PGL}

\DeclareMathOperator{\Pic}{Pic}

\newcommand{\slf}{\mathfrak{sl}}
\newcommand{\git}{/\!\!/}

\newcommand{\C}{\mathbb{C}}
\newcommand{\CC}{\mathbb{C}}

\newcommand{\IH}{IH}

\newcommand{\inv}{^{-1}}

\newcommand{\RR}{\mathbb{R}}
\newcommand{\ZZ}{\mathbb{Z}}

\begin{document}

\title[Intersection cohomology of higher rank Teichm\"uller components]{Intersection cohomology of  higher rank Teichm\"uller components}

\author[Mathieu Ballandras]{Mathieu Ballandras}
\address{Instituto de Ciencias Matem\'aticas \\
  CSIC-UAM-UC3M-UCM \\ Nicol\'as Cabrera, 13--15 \\ 28049 Madrid \\ Spain}
\email{mathieu.ballandras@icmat.es}

 \author[Oscar Garc{\'\i}a-Prada]{Oscar Garc{\'\i}a-Prada}
\address{Instituto de Ciencias Matem\'aticas \\
  CSIC-UAM-UC3M-UCM \\ Nicol\'as Cabrera, 13--15 \\ 28049 Madrid \\ Spain}
\email{oscar.garcia-prada@icmat.es}

\noindent
\thanks{
 \noindent
Partially supported by the Spanish Ministry of Science and
Innovation, through the ``Severo Ochoa Programme for Centres of Excellence in R\&D (CEX2019-000904-S)'' and grant  PID2019-109339GB-C31}


\begin{abstract}
Exploiting the non-abelian Hodge correspondence, together with the Cayley correspondence, in this paper, we compute the intersection  cohomology of certain singular higher rank  Teichm\"uller components of  character varieties of the fundamental group of a compact Riemann surface. 

\end{abstract}

\maketitle

\section{Introduction}

Let $\mathcal{R}(G)$ be the character variety classifying reductive representations of the fundamental group of a compact Riemann surface of genus $g\geq 2$ in a real semisimple Lie group $G$. Higher rank Teichm\"uller components are connected components of $\mathcal{R}(G)$ consisting entirely of faithful and discrete representations. In this article, we compute the intersection cohomology of certain higher rank Teichm\"uller components using the Higgs bundle point of view initiated by Hitchin \cite{Hitchin}.

The non-abelian Hodge correspondence provides a powerful approach to studying the topology of character varieties. This correspondence, due to Hitchin \cite{Hitchin}, Donaldson \cite{Donaldson}, Simpson \cite{Simpson}, Corlette \cite{Corlette}  and others
(see \cite{Garcia-Prada_Gothen_Mundet} in particular for the case of real groups), gives a homeomorphism between the character variety $\mathcal{R}(G)$ and the
moduli space $\mathcal{M}_K(G)$ of $G$-Higgs bundles. Here $K$ is the canonical line bundle of the Riemann surface. The complex structure of the Riemann surface endows  $\mathcal{M}_K(G)$ with a very rich structure, in particular with an action of $\CC^*$, which makes the study of the topology of $\mathcal{M}_K(G)$ much more  amenable than the direct study on the character variety. 
This approach has been very successful in the  computation of topological invariants of character varieties such as the number of connected components, the cohomology or the intersection cohomology.

In this paper, in addition to the non-abelian Hodge correspondence, we also take advantage of  the  Cayley correspondence that exists for certain classes of real groups. One such class is given by the Hermitian groups of tube type. In this situation the Cayley correspondence was studied on a case by case fashion for the classical groups in \cite{Bradlow_Garcia-Prada_Gothen,BGG06,BGG15,GGM13} and in general in \cite{BGR17}. Another class is given by the special orthogonal groups. Collier \cite{Collier} studied the case of $\SO(n,n+1)$ and the general $\SO(p,q)$ case  was studied in \cite{Aparicio-Arroyo_Bradlow_Collier_Garcia-Prada_Gothen_Oliveira}. A general Cayley correspondence has been given more recently in \cite{Bradlow_Collier_Garcia-Prada_Gothen_Oliveira}, including the previous cases, as well as the case of quaternionic real forms of $F_4$, $E_6$, $E_7$ and $E_8$. This general construction includes also the construction of the Hitchin components for split real forms given by Hitchin \cite{Hitchin92}. 
For a group $G$ belonging to one of the  families mentioned above, the Cayley correspondence gives a description of some topological  components of the moduli space $\mathcal{M}_K(G)$ in terms of a moduli space of $K^m$-twisted $G'$-Higgs bundles  for another group $G'$ (for more details see Definition \ref{def_twisted_higgs}, Theorem \ref{th_cayley} and \cite{Bradlow_Collier_Garcia-Prada_Gothen_Oliveira}). The connected components of  $\mathcal{M}_K(G)$ described in this way are called Cayley components.

This correspondence is particularly interesting from the point of view of higher rank Teichm\"uller theory. Labourie \cite{Labourie} introduced the notion of Anosov representation, and Guichard--Labourie--Wienhard \cite{Guichard_Labourie_Wienhard}, building upon the concept of positive structure given by Guichard--Wienhard \cite{Guichard_Wienhard,Guichard_Wienhard2},  introduced the notion of positive
representations. Positive representations are in particular Anosov and hence
 discrete and faithful.
The components of the character varieties corresponding to Cayley components consist entirely of positive  representations,
which means that they are higher rank Teichm\"uller components, and therefore objects of central interest from the point of view of higher rank Teichm\"uller theory.

In this article we exploite the Cayley correspondence to compute the intersection cohomology of certain singular higher rank Teichm\"uller components for the following pairs of groups
\begin{enumerate}
    \item $G=\SO(n,n+1)$, $G'=\SO(1,2)$
    \item $G=\SO_0(n,n+2)$, $G'=\SO_0(1,3)$ for $n$ odd
    \item $G=\PU(n,n)$, $G'=\PGL(n,\CC)$ and $G=\U(n,n)$, $G'=\GL(n,\CC)$
    \item $G$ the quaternionic real form of the adjoint form of the complex group $\E_6$, and $G'=\PGL(3,\CC)$.
\end{enumerate}

The first case is treated in Section \ref{section_SO}. The Cayley components were described by Collier \cite{Collier}. When $n=2$ they were already studied by Gothen \cite{Gothen01} in the closely related situation of $\Sp(4,\RR)$-Higgs bundles. The relation between both points of view comes from the isomorphism between $\PSp(4,\RR)$ and the component of the identity of $\SO(2,3)$. The corresponding Cayley components, referred  as  Collier--Gothen components,  form an open and closed subvariety of $\mathcal{M}_K\left(\SO(n,n+1)\right)$ isomorphic to
\[
\mathcal{M}_{K^n}\left(\SO(1,2)\right)\times \bigoplus_{j=1}^n H^0(K^{2j}).
\]

For  $n\ge 2$, only one of these components, denoted by
$\mathcal{M}^{\mathcal{C},0}_K\left(\SO(n,n+1)\right)$
has singularities `worse' than finite quotient singularities, and it is hence 
interesting to study its intersection cohomology. We compute its Hodge polynomial for intersection cohomology with the general method developped by Kirwan \cite{Kirwan1986}.

In the other cases the intersection cohomology can be obtained by combining various previous results about moduli spaces of Higgs bundles. The first important result, due to Maulik and Shen \cite{Maulik_Shen}, is the fact that the intersection cohomolgy of the moduli space of $L$-twisted Higgs bundles of rank $n$ and degree $d$ is independent of $d$ when $\deg L > 2g -2$. In particular, by taking $n$ and $d$ coprime, the moduli space is smooth, and hence the intersection cohomology is nothing but the usual cohomology. The cohomology of this moduli space was computed by counting Higgs bundles over finite field by Schiffmann \cite{Schiffmann}, Mozgovoy--Schiffmann \cite{Mozgovoy_Schiffmann} and Mozgovoy--O'Gorman \cite{Mozgovoy_OGorman}. This is  explained in Section \ref{sect_ic_higgs_bundles}.

Using these results we take care of $\U(n,n)$ of case (3) in Section \ref{unn}. To deal with  $\PU(n,n)$ of case (3), in Section \ref{punn} we  exploit a well-known relation between the moduli spaces of $\GL(n,\CC)$-Higgs bundles and $\PGL(n,\CC)$-Higgs bundles to relate their intersection cohomology --- this is explained in Section \ref{gltopgl}. This also solves cases (4) in Section \ref{e6}
and (2) in Section \ref{sonn+2}. For (2), we use the isomorphism  $\SO_0(1,3)\cong \PGL(2,\CC)$.
We have not been able to deal with the  group $G=\SU(n,n)$, closely related to the groups in  case (3). This is due to the fact that its Cayley partner group  is $G'=\SL_n(\CC)\rtimes \RR^*$ (see \cite{BGR17}), and  we do not know how to compute the intersection cohomology of $\mathcal{M}_{K^2}(G')$.

\section{Character varieties and moduli spaces of Higgs bundles}
In this section $\Sigma_g$ is a compact smooth surface of genus $g\ge 2$ and $G$  a real reductive group. In Theorems \ref{th_nah_semisimple} and \ref{th_cayley} it is moreover assumed that $G$ is semisimple.
\subsection{Character varieties}
A representation of $\pi_1(\Sigma_g)$ in $G$ is reductive if and only if its composition with the adjoint action of $G$ on its Lie algebra $\mathfrak{g}$ decomposes as a direct sum of irreducible representations. The set of reductive representations is denoted by $\Hom^{\rd}\left(\pi_1(\Sigma_g), G\right)$, the group $G$ acts on this set by conjugation.
\begin{definition}
    The {\em character variety} or {\em moduli space of representations} $\mathcal{R}(G)$ is defined by:
    \[
    \mathcal{R}(G) := \Hom^{\rd}\left(\pi_1(\Sigma_g),G\right) / G.
    \]
\end{definition}
\begin{remark}
Restricting to reductive representations ensures in particular that $\mathcal{R}(G)$ is Hausdorff.
\end{remark}
There is an equivalence of categories between the category of representations of $\pi_1(\Sigma_g)$ in $G$, and the category of principal $G$-bundles on $\Sigma_g$ endowed with a flat connection. The topological class of principal $G$-bundles induces an open and closed decomposition of the character variety. If $G$ is connected, the topological classes are indexed by the fundamental group of $G$ and
\[
\mathcal{R}(G) = \bigsqcup_{\gamma\in \pi_1(G)} \mathcal{R}^\gamma(G).
\]
Therefore the character variety $\mathcal{R}(G)$ is generally not connected. Moreover the $\mathcal{R}^\gamma(G)$ themselves are not always connected. In this article we study particular connected component of $\mathcal{R}(G)$ called higher rank Teichm\"uller components. 
\begin{definition}
    A {\em higher rank Teichm\"uller component} is defined as a connected component of the character variety $\mathcal{R}(G)$ consisting entirely of representations which are faithful with discrete image.
\end{definition}
\begin{remark}
    The usual Teichm\"uller space of $\Sigma_g$ is such a component in $\mathcal{R}(\PSL(2,\RR))$.
\end{remark}

\subsection{Higgs bundles}
Let $X$ be a complex smooth projective curve of genus $g\ge 2$, whose underlying smooth surface  is the compact surface $\Sigma_g$. Let $H$ be a maximal compact subgroup of $G$. The Lie algebra of $G$ admits a Cartan decomposition $\mathfrak{g}=\mathfrak{h}\oplus \mathfrak{m}$ with $\mathfrak{h}$ the Lie algebra of the maximal compact subgroup $H$. Tensoring over $\RR$ with $\CC$ provides a decomposition of the complexified Lie algebra $\mathfrak{g}^{\CC}=\mathfrak{h}^{\CC}\oplus \mathfrak{m}^{\CC}$. The group $H^{\CC}$, obtained by complexifying $H$, acts on $\mathfrak{m}^{\CC}$ by restriction of the adjoint action.
\begin{definition}
A $G$-{\em Higgs bundle} is a pair  $(E,\phi)$ consisting of 
a holomorphic principal $H^{\mathbb{C}}$-bundle $E$ on $X$ and  a Higgs field $\phi \in H^0\left(X, K\otimes E(\mathfrak{m}^{\CC})\right)$. Here 
$K$ is the canonical bundle of $X$ and $E(\mathfrak{m}^{\CC}) := E\times \mathfrak{m}^{\CC}/H^{\CC}$.
\end{definition}
\begin{remark}
    For $G=\GL(n,\CC)$, one can exploit the equivalence between  principal $\GL(n,\CC)$-bundles  and rank $n$ vector bundles, and consider  $E$ to be a holomorphic vector bundle of rank $n$. Then the Higgs field lies in $H^0\left(X,K\otimes \End(E)\right)$.  This is the original  notion of Higgs bundle introduced by Hitchin in his seminal paper \cite{Hitchin}.
\end{remark}
There is a suitable notion of stability for $G$-Higgs bundles which allows to construct a {\em moduli space}  $\mathcal{M}_K\left(G\right)$ classifying polystable $G$-Higgs bundles up to isomorphism. This is an algebraic variety over $\CC$ \cite{Garcia-Prada_Gothen_Mundet}.

\begin{remark}
The topological type of the principal $H^{\CC}$-bundle $E$ induces an open and closed decomposition of the moduli space. If $G$ is connected, so is $H$ and the topological types are indexed by $\pi_1(H^{\CC}) \cong \pi_1(G)$ 
\[
\mathcal{M}_K\left(G\right) \cong \bigsqcup_{\gamma\in \pi_1(G)} \mathcal{M}^\gamma_K\left(G\right).
\]
\end{remark}
Higgs bundles offer an interesting insight to study character varieties thanks to the non-abelian Hodge correspondence. This theory due to Hitchin \cite{Hitchin}, Donaldson \cite{Donaldson}, Simpson \cite{Simpson}, Corlette \cite{Corlette}  (for general real forms see \cite{Garcia-Prada_Gothen_Mundet}) culminates in the following results.
\begin{theorem}\label{th_nah_semisimple}
    For $G$ semisimple, the character variety $\mathcal{R}(G)$ is homeomorphic to the moduli space $\mathcal{M}_K(G)$. Moreover, this homeomorphism respects topological types and hence it induces an homeomorphism between $\mathcal{R}^\gamma(G)$ and $\mathcal{M}^\gamma_K(G)$.
\end{theorem}

For $G$ a real reductive group there are two possibilities to obtain a similar statement. The first possibility (Theorem \ref{th_torions_nah}) is to restrict to components with torsion topological types. The second possibility is to consider representations of a larger group than $\pi_1(\Sigma_g)$, namely a central extension $\Gamma$ of $\pi_1(\Sigma_g)$ (see Theorem \ref{th_gamma_nah}). 
\begin{theorem}\label{th_torions_nah}
    For $G$ a connected real reductive group and $\gamma$ a torsion element in $\pi_1(G)$, the component $\mathcal{R}^{\gamma}(G)$ of the character variety is homeomorphic to $\mathcal{M}^{\gamma}_K(G)$.
\end{theorem}
The second point of view is detailed in Atiyah--Bott \cite[Section 6]{Atiyah_Bott} for $G$ compact. As in this reference, consider a central extension $\Gamma$ of $\pi_1(\Sigma_g)$ with centre $\RR$. Let $\Hom^{Z,\rd}\left(\Gamma,G\right)$ be the set of reductive representations $\rho : \Gamma \to G$ such that $\rho(\RR) \subset Z(G)$ and the following diagram commutes
\[
\begin{tikzcd}
    0  \arrow[r] & \RR \arrow[d]\arrow[r] & \Gamma \arrow[d,"\rho"]\arrow[r]& \pi_1(\Sigma_g) \arrow[d]\arrow[r] & 1 \\
    1 \arrow[r] & Z(G) \arrow[r] & G \arrow[r] & G/Z(G) \arrow[r] & 1.
\end{tikzcd}
\]
The relevant character variety for the non-abelian Hodge correspondence is now $\widetilde{\mathcal{R}}(G) := \Hom^{Z,\rd}\left(\Gamma,G\right) / G $. 
\begin{theorem}\label{th_gamma_nah}
    For $G$ a real reductive group, the non-abelian Hodge correspondence provides a homeomorphism between $\widetilde{\mathcal{R}}(G)$ and $\mathcal{M}_K(G)$.
\end{theorem}
\begin{remark}
    We mention this version of the non-abelian Hodge correspondence for reductive groups since we will be dealing with  the case $G=\U(n,n)$. 
\end{remark}
One  upshot of the non-abelian Hodge correspondence is that it provides many powerful tools for the study the topology of character varieties. Indeed, the complex structure of the Riemann surface endowes  $\mathcal{M}_K(G)$ with a very rich structure, in particular with an action of $\CC^*$ scaling the Higgs field, which makes the study of the topology of $\mathcal{M}_K(G)$ much more  amenable than the direct study on the character variety.  This action combined with Morse-theoretic arguments and similar algebraic localization arguments can be used to count connected components and also to compute cohomology (see e.g.
\cite{Hitchin, Garcia-Prada_Heinloth,Garcia-Prada_Heinloth_Schmitt}).

\subsection{Higgs bundles and higher rank Teichm\"uller components}
In order to describe the Cayley correspondence, which plays a central role in this work, we need a slight generalization of the notion of Higgs bundles where the canonical bundle $K$ is replaced by any line bundle $L$.

\begin{definition}\label{def_twisted_higgs}
  Let $L$ be a holomorphic line bundle on $X$. An $L$-{\em twisted $G$-Higgs bundle} is a pair $(E,\phi)$ consisting of
  a principal $H^{\mathbb{C}}$-bundle $E$ on $X$ and Higgs field $\phi \in H^0\left(X, L\otimes E(\mathfrak{m}^{\CC})\right)$.
  
The moduli space of polystable $L$-twisted $G$-Higgs bundles is a complex algebraic variety denoted by $\mathcal{M}_L(G)$.
\end{definition}

\begin{theorem}[Cayley correspondence \cite{Bradlow_Collier_Garcia-Prada_Gothen_Oliveira}]\label{th_cayley}
    If $G$ is semisimple and is either
    \begin{enumerate}
        \item a split real form,
        \item Hermitian of tube type,
        \item locally isomorphic to $\SO(p,q)$,
        \item a quaternionic real form of $\FF_4$, $\E_6$, $\E_7$ or $\E_8$,
    \end{enumerate}
then there exists a real semisimple group $G'$, integers $N$, $(l_i)_{1\le i \le N}$, $m_c$, and an open and closed embedding $\Psi$
    \[
    \Psi : \mathcal{M}_{K^{m_c+1}}(G') \times \bigoplus^N_{i = 1} H^0(X,K^{l_i+1}) \to \mathcal{M}_K(G).
    \]
    The group $G'$ and the integers $N$, $l_i$ and $m_c$ admit an explicit Lie theoretic description detailed in \cite{Bradlow_Collier_Garcia-Prada_Gothen_Oliveira}.
\end{theorem}

    The map $\Psi$ is called the {\em Cayley map}, the connected components of $\mathcal{M}_K(G)$ in the image of the Cayley map are called {\em Cayley components}.

    This Cayley correspondence sheds a new light on the moduli space of Higgs bundles and therefore on the character varieties. In particular,  this correspondence generally detects {\em a priori} hidden topological invariants associated to the group $G'$. 
    For example, it was used in \cite{Aparicio-Arroyo_Bradlow_Collier_Garcia-Prada_Gothen_Oliveira} to identify connected components thanks to the topological invariant associated to $G'$. In general, the moduli space $\mathcal{M}_{K^n}(G)$ for $n\ge 2$ is
    `less singular' than $\mathcal{M}_{K}(G)$, due to the vanishing of the hypercohomology in degree two of the deformation complex. Therefore Cayley components are generally less singular than general components.

From the point of view of the character variety and higher rank Teichm\"uller theory, this correspondence is particularly interesting. Labourie \cite{Labourie} introduced the notion of Anosov representations, Guichard--Wienhard \cite{Guichard_Wienhard} introduced the notion of positivity and Guichard--Labourie--Wienhard \cite{Guichard_Labourie_Wienhard} the notion of positive  representations. Positive representations are Anosov and hence  discrete and faithful.
It is proved in \cite{Bradlow_Collier_Garcia-Prada_Gothen_Oliveira} that Cayley components (seen as components of the character varieties thanks to the non-abelian Hodge correspondence) contain positive Anosov representations. This, combined with other results in  \cite{Bradlow_Collier_Garcia-Prada_Gothen_Oliveira} and \cite{Guichard_Labourie_Wienhard}, implies that the components of the character variety corresponding to the Cayley components are higher rank Teichm\"uller components.

\section{Intersection cohomology}\label{sect_ic}
Intersection cohomology is a cohomology theory well suited to study singular varieties (see  Beilinson--Bernstein--Deligne--Gabber \cite{Beilinson_Bernstein_Deligne_Gabber}). To describe it, let $Y\xrightarrow{p} \Spec \CC$ be an irreducible variety. There exists a particular object in the bounded derived category  $\mathcal{D}^c_b(Y)$ of constructible sheaves on $Y$ called the intersection complex, which is  denoted by $\mathcal{IC}_Y$. Consider the  derived push-forward functor  $R p_*$, and the proper push-forward functor  $R p_!$.
The intersection cohomology, respectively compactly supported intersection cohomology,  is obtained by taking $R p_* \mathcal{IC}_Y [\dim Y]$, respectively $R p_! \mathcal{IC}_Y [\dim Y]$. These are elements in $\mathcal{D}^c_b\left(\Spec \CC\right)$, and are identified with graded complex vector spaces $\oplus_k IH^k(Y,\CC)$, respectively $\oplus_k IH_c^k(Y,\CC)$, with a mixed-Hodge structure. Recall that a mixed Hodge structure is given by a finite increasing filtration, called the weight filtration, with a pure  Hodge structure  on the complexified subquotients of this filtration (see \cite{saito}).
When the variety is smooth, the intersection complex is, up to a shift, a constant sheaf so that the intersection cohomology coincides with the usual cohomology. The mixed-Hodge numbers of those structures are encoded in the mixed-Hodge polynomial 
\[
MH\left(Y; u, v, t\right) = \sum_{i,j,k} \dim \IH^{i,j,k}\left(Y, \CC\right) u^i v^j t^k,
\]
respectively the compactly supported mixed-Hodge polynomial $MH_c\left(Y; u, v, t\right)$,
which admit as  important specializations  the $E$-polynomial
\[
E\left(X; u, v \right) := MH_c\left(Y; u, v, -1\right), 
\]
and the Poincar\'e polynomial
\[
P\left(Y; t\right) := MH\left(Y; 1, 1, -t\right) = \sum_k (-1)^k\dim IH^k\left(Y, \CC\right)t^k,
\]
respectively the Poincar\'e polynomial $P_c\left(Y; t\right)$ for compactly supported intersection cohomology.
The intersection cohomology of $Y$ is said to be pure if $IH^{i,j,k}\left(Y,\CC\right) = 0$ for $i + j \ne k$. This is the case for instance if $Y$ is a projective variety. In this case we can set $t=1$ in the mixed-Hodge polynomial without losing any information, 
\[
H(Y; u, v) := MH(Y; u, v, 1)
\]
is then called the Hodge polynomial. 

In some contexts, the intersection cohomology seems to be the natural generalization to singular varieties of the usual cohomology. For instance it still satisfies Poincar\'e duality.

\section{\texorpdfstring{Intersection cohomology of the moduli space of $K^l$-twisted Higgs bundles}{Intersection cohomology of the moduli space of twisted Higgs bundles}} \label{sect_ic_higgs_bundles}
In this section we recall previous results about cohomology and intersection cohomology of the moduli spaces of $K^l$-twisted Higgs bundles and $\PGL(n,\CC)$-Higgs bundles. The moduli spaces describing Cayley components all involve a twisting by $K^l$ with $l>1$, hence we also assume $l>1$ in this section.
\subsection{From cohomology to intersection cohomology}\label{subsect_ic_inv_d}
The connected component of the moduli space of $K^l$-twisted $\GL(n,\CC)$-Higgs bundles are indexed by the degree
\[
\mathcal{M}_{K^l}\left(\GL(n,\CC)\right) = \bigsqcup_{d\in \ZZ} \mathcal{M}^d_{K^l}\left(\GL(n,\CC)\right),
\]
and similarly for the moduli space of $\PGL(n,\CC)$--Higss bundles
\[
\mathcal{M}_{K^l}\left(\PGL(n,\CC)\right) = \bigsqcup_{d\in \ZZ/ n\ZZ} \mathcal{M}^d_{K^l}\left(\PGL(n,\CC)\right).
\]
The cohomology of the moduli space of Higgs bundles of rank $n$ and degree $d$ has been extensively studied, in particular for $n$ and $d$ coprime, in which case the moduli space  is smooth. Recursive formulas for arbitrary rank were given in \cite{Garcia-Prada_Heinloth,Garcia-Prada_Heinloth_Schmitt}. The first closed formula in any rank was obtained by Schiffmann \cite{Schiffmann}, and was generalized to twisted Higgs bundles by Mozgovoy--Schiffmann \cite{Mozgovoy_Schiffmann}. The proof relies on the Weil conjectures, and the cohomology is obtained by counting points of the moduli space over finite fields. The formula from Mozgovoy--Schiffmann was simplified by Mozgovoy--O'Gorman \cite{Mozgovoy_OGorman}. Before detailing this result we explain how to go from cohomology of the moduli space of $\GL(n,\CC)$-Higgs bundles in the coprime case to intersection cohomology in the general case for both $\GL(n,\CC)$ and $\PGL(n,\CC)$. Here the assumption $l>1$ is crucial, it is an instance of the general principle that moduli spaces of $L$-twisted Higgs bundles are better behaved and less singular when $\deg L > \deg K$. Since $\deg K^l > \deg K$, by a result of Maulik--Shen \cite{Maulik_Shen}, the intersection cohomology of $\mathcal{M}^d_{K^l}\left(\GL(n,\CC)\right)$ does not depend on the degree
\begin{equation*}
IH^*\left(\mathcal{M}^d_{K^l}\left(\GL(n,\CC)\right),\CC\right) \cong IH^*\left(\mathcal{M}^{d'}_{K^l}\left(\GL(n,\CC)\right),\CC\right) \text{ for all }d,\\ d'\in\ZZ.
\end{equation*}
In particular we can chose $d'$ coprime with $n$, then the moduli space is smooth and its intersection cohomology coincides with the usual cohomology
\[
IH^*\left(\mathcal{M}^d_{K^l}\left(\GL(n,\CC)\right),\CC\right) \cong H^*\left(\mathcal{M}^{d'}_{K^l}\left(\GL(n,\CC)\right),\CC\right)
\]
for any $d\in\ZZ$ and for any $d'$ coprime with $n$.
This is an example where the intersection cohomology seems clearly to be the natural cohomology to study singular varieties.
\subsection{\texorpdfstring{From $\GL(n,\CC)$ to $\PGL(n,\CC)$}{From GL_n to PGL_n}}\label{gltopgl}
For the reader's convenience we recall a well known result relating moduli spaces of $\GL(n,\CC)$-Higgs bundles and moduli spaces of $\PGL(n,\CC)$-Higgs bundles, it is explained by Mauri \cite{Mauri} and Mauri--Felisetti \cite{Mauri_Felisetti} in the degree $0$ case.
\begin{lemma}\label{lemma_PGL_GL}
    The intersection cohomology of the moduli spaces of $L$-twisted $\GL(n,\CC)$-Higgs bundles and $\PGL(n,\CC)$-Higgs bundles are related by
    \[
    IH^*\left(\mathcal{M}^d_{L}\left(\GL(n,\CC)\right),\CC\right)\cong IH^*\left(\mathcal{M}^d_{L}\left(\PGL(n,\CC)\right),\CC\right)\otimes H^*\left(\Pic^0(X),\CC\right)
    \]
    and similarly for compactly supported intersection cohomology.
\end{lemma}
\begin{proof}
Consider the map
\[
\begin{array}{cccc}
      \pi & : \mathcal{M}^d_L\left(\GL(n,\CC)\right) &\to & \Pic^d(X)\times H^0(X,L)   \\
    & (E,\phi) & \mapsto &  \left(\det E, \tr \phi\right)
\end{array}
\]
and a fibre $\mathcal{F}^{d,n}_L(M) := \pi^{-1}\left(M,0\right)$ with $M$ a line bundle of degree $d$. This the moduli space of $L$-twisted Higgs bundles with fixed determinant and traceless Higgs field. Notice that when $d=0$ this fibre is isomorphic to the moduli space $\mathcal{M}_L\left(\SL_n(\CC)\right)$. The map $\pi$ is an étale-locally trivial fibration, indeed the following diagram is cartesian
\[
\begin{tikzcd}
\mathcal{F}^{d,n}_L(M)\times \Pic^0(X)\times H^0(X,L) \arrow[r,"p'"] \arrow[swap,d,"\pi'"]&\mathcal{M}^d_L\left(\GL(n,\CC)\right)
\arrow[d, "\pi"]\\
\Pic^0(X)\times H^0(X,L)\arrow[swap,r,"p"]& \Pic^d(X)\times H^0(X,L) \\
\end{tikzcd}
\]
with $\pi'$ the projection, $p(N,s) := N^n\otimes M$ and $p'\left((E,\phi),N,s\right) := \left(E\otimes N,\phi + \frac{s}{n}\Id_E\right)$. The map $p'$ is a Galois cover with group $\Gamma \subset \Pic^0(X)$ the group of line bundles of order $n$ acting by
\[
\gamma.\left((E,\phi), N , s\right) := \left((\gamma\otimes E, \phi), \gamma^{-1}\otimes N, s\right).
\]
Therefore the intersection cohomology of $\mathcal{M}^d_L\left(\GL(n,\CC)\right)$ is the $\Gamma$-invariant part of the intersection cohomology of the Galois cover (see Kirwan \cite[Lemma 2.12]{Kirwan1986}),
\begin{equation*}
IH^*\left(\mathcal{M}^d_L\left(\GL(n,\CC)\right),\CC\right)=IH^*\left(\mathcal{F}^{d,n}_L(M)\times \Pic^0(X)\times H^0(L),\CC\right)^{\Gamma}.
\end{equation*}
The $\Gamma$-action on $H^*\left(\Pic^0(X),\CC\right)$ is a restriction of a $\Pic^0(X)$-action hence is trivial and
\[
IH^*\left(\mathcal{M}^d_L\left(\GL(n,\CC)\right),\CC\right) = IH^*\left(\mathcal{F}^{d,n}_L(M),\CC\right)^{\Gamma}\otimes H^*\left(\Pic^0(X),\CC\right).
\]
The conclusion follows noticing $\mathcal{F}^{d,n}_L(M)/\Gamma \cong \mathcal{M}_L\left(\PGL(n,\CC)\right) $.
\end{proof}
\begin{remark}\label{remark_Hodge_pol}
    The intersection cohomology of $\mathcal{M}^d_{K^l}\left(\GL(n,\CC)\right)$ is pure (see Mauri \cite[Proposition 2.4]{Mauri}). Lemma \ref{lemma_PGL_GL} has the following consequence in terms of Hodge polynomial for intersection cohomology
    \[
    H\left(\mathcal{M}^d_{K^l}\left(\GL(n,\CC)\right); u,v\right) =     H\left(\mathcal{M}^d_{K^l}\left(\PGL(n,\CC)\right); u,v\right)     H\left(\Pic^0(X); u,v\right). 
    \]
    Also true for compactly supported intersection cohomology.
   The cohomology of $\Pic^0(X)$ is well-known and hence
    \[
        H\left(\mathcal{M}^d_{K^l}\left(\PGL(n,\CC)\right); u,v\right) = \frac{H\left(\mathcal{M}^d_{K^l}\left(\GL(n,\CC)\right); u,v\right)  }{(1+u)^g (1+v)^g}.
    \]
\end{remark}

\subsection{From point counting to cohomology}
As mentioned above, in the coprime case, the cohomology of the moduli space of (twisted) Higgs bundles has been computed by counting points over finite fields by Mozgovoy--O'Gorman \cite{Mozgovoy_OGorman} relying on previous work by Mozgovoy--Schiffmann \cite{Mozgovoy_Schiffmann} and Schiffmann \cite{Schiffmann}. 

In the remaining of this section we recall the formula they obtained in the particular case of a twisting by $K^l$ with $l>1$. The formula involves generating series defined by summing over partitions. We denote by $\mathcal{P}$ the set of partitions of integers. For a partition $\lambda$ we denote by $d(\lambda)$ its Young diagramm. For instance $\lambda = (5,4,2)$ is a partition of $11 = |\lambda|$ with Young diagram
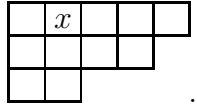
\[
\begin{Young}
      & $x$ & &  &  \cr
      &  &  & \cr
      & \cr
\end{Young}.
\]
For a box $x$ in a Young diagram $d(\lambda)$, we denote by $a(x)$ the arm length and by $l(x)$ the leg length, for instance in the previous diagram $a(x)=3$ and $l(x)=2$.
In order to count Higgs bundles we work over a finite field $\mathbb{F}_q$, and $X$ is a curve of genus $g$ defined over $\mathbb{F}_q$.
Let $(\alpha_i)_{1\le i \le g}$ be the Weil numbers of $X$. Consider the generating series
\begin{multline*}
\Omega(T,z) := \\ \sum_{\lambda \in \mathcal{P}} T^{|\lambda|} \prod_{x\in d(\lambda)}(-q^{a(x)}z^{l(x)})^{(l-1)(2 g - 2)}\frac{\prod_{i=1}^g(q^{a(x)}-\alpha_i^{-1} z^{l(x)+1})(q^{a(x)+1}-\alpha_i z^{l(x)})}{(q^{a(x)}-z^{l(x)+1})(q^{a(x)+1}-z^{l(x)})},
\end{multline*}
and
\[
\mathbb{H}(T,z) := (q-1)(1-z) \Log \Omega(T,q,z).
\]
The symbol $q$ can be understood either as the cardinal of the field $\mathbb{F}_q$ or as a formal variable. For a definition of the plethystic logarithm see \cite[Section 2]{Mozgovoy_OGorman}. After setting $z=1$, $\mathbb{H}(T,z)$ becomes the generating series for the Donaldson--Thomas invariants $\Omega_{n}$ of the moduli space of $K^l$-twisted Higgs bundles (\cite{KS})
\[
\mathbb{H}(T,1) = \sum_{n} q^{(1-l)(g-1)n}\Omega_n T^n.
\]

Here $\Omega_n$ is a rational function in $q$, and all the Weil numbers $\alpha_i$ and $\alpha_i^{-1}$ of $X$.
In particular for $n$ and $d$ coprime, $q^{l(g-1)n^2}\Omega_n$ is  the number of $\mathbb{F}_q$-points of the moduli space of $K^l$-twisted Higgs bundles. In the spirit of Weil conjecture, this information about the number of points gives some cohomological information, here the compactly supported Poincar\'e polynomial. The proof is the same as in the usual Higgs bundles case (see Schiffmann \cite[Corollary 1.3]{Schiffmann}). 
\begin{theorem}[Mozgovoy--O'Gorman \cite{Mozgovoy_OGorman}]\label{th_Mozgovoy_OGorman}
    For $n$ and $d$ coprime, the Poincar\'e polynomial for compactly supported cohomology of $\mathcal{M}^d_{K^l}\left(\GL(n,\CC)\right)$ is obtained by setting $q=t^2$ and $\alpha_i=t$ for $1\le i \le g$ in $q^{l(g-1)n^2}\Omega_n$
    \[
    \sum_k (-1)^k\dim H_c^i\left(\mathcal{M}^d_{K^l}\left(\GL(n,\CC)\right)\right)t^k= {t^{2l(g-1)n^2}\Omega_n}_{\left|q = t^2,\  \alpha_i = t\right.}.
    \]
\end{theorem}
\begin{corollary}\label{cor_pgl}
    For any rank $n$ and any degree $d$, the Poincar\'e polynomial (see Section \ref{sect_ic}) for compactly supported \textit{intersection} cohomology of $\mathcal{M}^d_{K^l}\left(\GL(n,\CC)\right)$ is
    \[
    P_c\left(\mathcal{M}^d_{K^l}\left(\GL(n,\CC)\right)\right) = {t^{2l(g-1)n^2}\Omega_n}_{\left|q = t^2,\  \alpha_i = t\right.},
    \]
    and for the group $\PGL(n,\CC)$
    \[
    P_c\left(\mathcal{M}^d_{K^l}\left(\PGL(n,\CC)\right)\right) = \frac{{t^{2l(g-1)n^2}\Omega_n}_{\left|q = t^2,\  \alpha_i = t\right.}}{(t-1)^{2 g}}.
    \]
\end{corollary}
\begin{proof}
    The result follows from Theorem \ref{th_Mozgovoy_OGorman}, Section \ref{subsect_ic_inv_d} and Remark \ref{remark_Hodge_pol}.
\end{proof}

\section{\texorpdfstring{$\SO(p,q)$-Higgs bundles}{{SO(p,q)}-Higgs bundles}}
\subsection{Generalities}
Similar to the case of $\GL(n,\CC)$, for other  classical groups we can  work with vector bundles rather than
principal bundles.  For instance for the group $\SO(p,q)$ and for its identity component $\SO_0(p,q)$ one has the  the following definition.
\begin{definition}\label{def_sopq_Higgs}
  Let $L$ be a line bundle on $X$. An $L$-{\em twisted $\SO(p,q)$-Higgs bundle} is a triple $(V,W,\phi)$ consisting of
holomorphic vector bundles $V$ and $W$ of ranks $p$ and $q$ respectively,  endowed with  orthogonal structures, such that $\det W = \det V$, and  a Higgs field $\phi : V \to W \otimes L$. An $L$-{\em twisted $\SO_0(p,q)$-Higgs bundle} is an  $L$-{\em twisted $\SO(p,q)$-Higgs bundle} with $\det V = \det W = \mathcal{O}_X$.
\end{definition}
For $\SO(p,q)$, the Cayley map  established in  \cite{Aparicio-Arroyo_Bradlow_Collier_Garcia-Prada_Gothen_Oliveira} is given by
\begin{equation}\label{eq_cayley_map}
\Psi : \mathcal{M}_{K^p}\left(\SO(1, q-p+1)\right)\times \bigoplus_{j=1}^{p-1}H^0(K^{2j}) \to \mathcal{M}_{K}\left(\SO(p,q)\right).
\end{equation}
In the following we recall this correspondence without worrying about the Higgs fields and we also study the case of $\SO_0(p,q)$, this will be used later to study Cayley components for $\SO_0(n,n+2)$. A $K^p$-twisted $\SO(1,q-p+1)$-Higgs bundle is determined by a triple $(I, \widehat{W},\widehat{\eta})$, with $\widehat{W}$ a rank $q-p+1$ orthogonal bundle, $I$ is the line bundle $\det \widehat{W}$ and  $\widehat{\eta}$ is the Higgs field. To the class of such a twisted Higgs bundle, is associated a $\SO(p,q)$-Higgs bundle $(V,W,\eta)$ with
\begin{eqnarray*}
    V := I\otimes\mathcal{K}_{p-1}, \\
    W := \widehat{W} \oplus I\otimes \mathcal{K}_{p-2},
\end{eqnarray*}
where
\[
\mathcal{K}_n:=K^n\oplus K^{n-2}\oplus\dots \oplus K^{-n+2}\oplus K^{-n}.
\]
To prove that this induces an isomorphism at the level of components of moduli spaces it is checked in \cite{Aparicio-Arroyo_Bradlow_Collier_Garcia-Prada_Gothen_Oliveira} that $SO(1,q-p+1)$-gauge transformations of $(I, \widehat{W},\widehat{\eta})$ are in one to one correspondence with $SO(p,q)$-gauge transformation of $(V, W, \eta)$. The correspondence between gauge transformations is the following.

A $SO(1,q-p+1)$-gauge transformations $(\det g_{\widehat{W}},g_{\widehat{W}})$ is sent to the $SO(p,q)$-gauge transformation 
\[
\left(\det (g_{\widehat{W}}) \Id_V, \left(\begin{array}{cc}
    g_{\widehat{W}} &  0\\
     0 & \det (g_{\widehat{W}}) \Id_{\mathcal{K}_{p-2}} 
\end{array}{}\right)  \right).
\]

It is then argued that if $(g_V,g_W)$ is a $\SO(p,q)$-gauge transformation of $(V,W,\eta)$ then $g_W$ has the following form
\[
g_W = \left(\begin{array}{cc}
  g_{\widehat{W}}   & 0 \\
    0 & g_{\mathcal{K}_{p-2}}
\end{array}\right),
\]
with $g_{\widehat{W}}$ an orthogonal transformation of $\widehat{W}$ and $g_{\mathcal{K}_{p-2}}$ an orthogonal transformation of $\mathcal{K}_{p-2}$. Moreover they satisfy
\begin{equation}\label{eq_pm_id}
    (g_V,g_{\mathcal{K}_{p-2}}) = \pm (\Id_V,\Id_{\mathcal{K}_{p-2}}).
\end{equation}

When we consider the identity component $\SO_0(p,q)$ instead of the whole group $\SO(p,q)$ then the form of the Cayley correspondence depends on the parity of $p$.
\begin{proposition}\label{prop_sopq0}
For $p$ odd, the Cayley correspondence provides an open and closed embedding
\begin{equation}
\Psi_0 : \mathcal{M}_{K^p}\left(\SO_0(1, q-p+1)\right)\times \bigoplus_{j=1}^{p-1}H^0(K^{2j}) \to \mathcal{M}_{K}\left(\SO_0(p,q)\right).
\end{equation}
For $p$ even, the Cayley correspondence provides an open and closed embedding
\begin{equation}
\Psi_0 : \mathcal{M}_{K^p}\left(\SO(1, q-p+1)\right)\times \bigoplus_{j=1}^{p-1}H^0(K^{2j}) \to \mathcal{M}_{K}\left(\SO_0(p,q)\right).
\end{equation}
\end{proposition}
\begin{proof}
For $p$ even it follows from $(\det W)^2\cong \mathcal{O}_X$ that, with the previous construction, $(V,W)$ defines an $\SO_0(p,q)$-bundle. Moreover for a $\SO(p,q)$-gauge transformation $(g_V,g_W)$ of $(V,W)$, relation \eqref{eq_pm_id} implies that $\det g_V = 1$. Therefore $\SO(p,q)$-gauge transformation of $(V,W)$ are exactly $\SO_0(p,q)$-gauge transformations, and the Cayley components can be seen either as components of $\mathcal{M}_K\left(\SO(p,q)\right)$ or as components of $\mathcal{M}_K\left(\SO_0(p,q)\right)$.

For $p$ odd, consider the previous construction but starting with a twisted $\SO_0(1,q-p+1)$-Higgs bundle instead of a general twisted $\SO(1,q-p+1)$-Higgs bundle. then $I$ is trivial and $(V,W,\eta)$ defines a $\SO_0(p,q)$-Higgs bundle. Moreover a $\SO_0(1,q-p+1)$-gauge transformation $(\Id,g_{\widehat{W}})$ is sent to a $\SO_0(p,q)$-gauge transformation. Now consider a $\SO_0(p,q)$-gauge transformation $(g_V,g_W)$ of $(V,W,\eta)$. As $p$ is odd, the sign in equation \eqref{eq_pm_id} is necessarily a plus, and hence  $(g_V,g_W)$ comes from a $\SO_0(1,q-p+1)$-gauge transformation $(\Id_{\det\widehat W},g_{\widehat{W}})$.

\end{proof}
\begin{remark}
    For $p$ even, when we are interested in the Cayley components for the group $\SO_0(p,q)$, the Cayley partner is still the whole group $\SO(1,q-p+1)$ and not the identity component $\SO_0(1,q-p+1)$. There is still a map
    \begin{equation*}
\mathcal{M}_{K^p}\left(\SO_0(1, q-p+1)\right)\times \bigoplus_{j=1}^{p-1}H^0(K^{2j}) \to \mathcal{M}_{K}\left(\SO_0(p,q)\right).
\end{equation*}
    but it is not an embedding and its image consists only of some Cayley components, not all of them.
\end{remark}
\subsection{\texorpdfstring{$\SO(n,n+1)$-Higgs bundles}{SO(n,n+1)-Higgs bundles} } \label{section_SO}
We consider a particular component of the moduli space of $K^n$-twisted $\SO(1,2)$-Higgs bundles studied by Collier \cite{Collier}. By means of  the Cayley correspondence, this will provide information about the moduli space of $\SO(n,n+1)$-Higgs bundles. In particular for $n=2$ the description of the component is closely related to the work of Gothen \cite{Gothen01} on $\Sp(4,\RR)$-Higgs bundles and of Alessandrini--Collier \cite{Alessandrini_Collier} on $\PSp(4,\RR)$-Higgs bundles. The corresponding Cayley components are called Collier--Gothen components. The relation comes from the fact that the component of the identity in $\SO(2,3)$ is isomorphic to $\PSp(4,\RR)$. 

Specializing $p=1$ and $q=2$ in Definition \ref{def_sopq_Higgs}, a $K^n$-twisted $\SO(1,2)$-Higgs bundle is then determined by  a holomorphic orthogonal vector bundle $W$ of rank $2$ and a Higgs field 
\[
\phi : \det W \to W\otimes K^n.
\]
We single out a component of $\mathcal{M}_{K^n}\left(\SO(1,2)\right)$ first by fixing the first Stiefel--Whitney class of $W$ to be zero. Then 
\[
W\cong M \oplus M^{-1}
\]
for some line bundle $M$. The connected component we are interested in is obtained by furthermore fixing $\deg M = 0$. The Higgs field $\phi$ is now determined by two morphisms
\[
\mu : \mathcal{O}_X \to M^{-1}\otimes K^n \text{ and }\nu : \mathcal{O}_X \to M \otimes K^n.
\]
Let
\[
\mathcal{F}=\left\lbrace \left(M, \mu, \nu \right) \left| \right. M\in\Pic^0(X), \mu \in H^0\left(X,M\inv K^n\right), \nu \in H^0\left(X,MK^n) \right.\right\rbrace.
\]
\begin{remark}
  For $n \ge 2$, $\mathcal{ F}$ is the total space of a vector bundle of rank $2 D$
  on $\Pic^0(X)$  with
  \[
D = h^0(M^{-1}K^n) = h^0(M K^n) = (2n -1 ) ( g - 1) .
\]
However, when $n=1$ the dimension of the fibre is not constant over $\Pic^0(X)$, there is a jump at $M=\mathcal{O}_X$ the trivial line bundle,
\[
h^0(M^{-1}K)= \left\lbrace
\begin{array}{cc}
g & \text{ if } M = \mathcal{O}_X \\
g-1 & \text{ if } M \ne \mathcal{O}_X
\end{array}\right. .
\]
This illustrates the fact that twisting by $K^n$ with $n>1$ produces less singular moduli spaces.
\end{remark}
In the following we assume $n \ge 2$, and hence $\mathcal{F}$ is the total space of a vector bundle of rank $2 D$ on $\Pic^0(X)$. The relevant component is obtained by taking the quotient by an $O(2,\CC)$-action. The group $O(2,\CC)$ is generated by $\begin{pmatrix} \lambda & \\
& \lambda\inv \end{pmatrix}$ and $\begin{pmatrix} 0 & 1 \\ 1 & 0\end{pmatrix}$. The $O(2,\CC)$-action is defined by
\begin{equation}\label{eq_O2_action}
    \begin{pmatrix} \lambda & \\
& \lambda\inv \end{pmatrix} .\left(M,\mu,\nu\right) =  \left(M, \lambda \mu , \lambda\inv \nu\right)  \text{ and } \begin{pmatrix} 0 & 1 \\ 1 & 0\end{pmatrix} . \left(M,\mu,\nu\right) = \left(M\inv,\nu,\mu\right) .
\end{equation}
The component we are interested in is $\mathcal{F} \git O(2,\CC)$. This is a singular variety. Collier studied its cohomology by considering a homotopy equivalence with the quotient of $\Pic^0(X)$ under inversion. Intersection cohomology is not invariant under homotopy equivalence and hence it cannot be computed in the same way. Among the connected components of the moduli space of $K^n$-twisted $\SO(1,2)$-Higgs bundles, this component is the only one with singularities `worse' than finite quotient singularities. Hence it is the only Collier--Gothen component for which the intersection cohomology does not necessarily coincide with the usual cohomology.
\begin{remark}\label{remark_SO_0}
    If we consider $\SO_0(1,2)$-Higgs bundles instead of $\SO(1,2)$-Higgs bundles, the component we are interested in is $\mathcal{F} \git SO(2,\CC)$.
\end{remark}
In the following, we use Kirwan's method \cite{Kirwan84, Kirwan_desingularisation, Kirwan1986} to compute intersection cohomology of GIT quotients. The invariant computed is the Hodge polynomial. Kirwan's theory involves equivariant cohomology. For $G$ a reductive group acting on a variety $Y$, similarly to the Poincar\'e polynomial, the Poincar\'e serie for $G$-equivariant cohomology is
\[
P^G(Y)=\sum_{k} (-1)^k \dim H^k_G \left(Y, \CC\right)t^k.
\]

\begin{lemma}\label{lemma_ic_fibre}
Let $V^+$ and $V^-$ be complex vector spaces of dimension $D$. Let $\C^*$ acts on $V^+ \oplus V^-$ by
\[
\lambda . (v^+, v^-) = (\lambda v^+, \lambda^{-1} v^-).
\]
The intersection cohomology of the GIT quotient $\left(V^+\oplus V^-\right)\git\CC^*$ is pure and its Hodge polynomial is
\[
H\left(\left(V^+\oplus V^-\right)\git\CC^* ; u ,v\right) = \frac{1-(uv)^{D}}{1-uv}.
\]
\end{lemma}
\begin{proof}
We start by computing the Poincar\'e polynomial. Following Kirwan \cite{Kirwan1986}, $IH^*\left(\left(V^+\oplus V^-\right)\git\CC^*,\CC\right)$ can be computed from the intersection cohomology of a quotient of a projective space $IH^*\left(\mathbb{P}\left(V^+\oplus V^-\right)\git\CC^*,\CC\right)$. For the lifted action of $\CC^*$ on $\mathbb{P}\left(V^+\oplus V^-\right)$, the semistable locus is the same as the stable locus and is the subvariety
\[
\mathbb{P}\left(V^+\oplus V^-\right)^{\text{ss}}=\left\lbrace [v^+,v^-] \left| v^+\ne 0 \text{ and } v^- \ne 0\right.\right\rbrace.
\]
The quotient $\mathbb{P}\left(V^+\oplus V^-\right)\git\CC^*$ is smooth so that its intersection cohomology coincides with the usual cohomology and it can be computed with \cite{Kirwan84}. As the action on the stable locus is free, the equivariant cohomology of this locus coincides with the cohomology of the quotient. Therefore the Poincar\'e polynomial of $\mathbb{P}\left(V^+\oplus V^-\right)\git\CC^*$ is given by
\begin{align*}
P\left(\mathbb{P}\left(V^+\oplus V^-\right)\git\CC^*\right) =&  P^{\CC^*}\left( \mathbb{P}\left(V^+\oplus V^-\right) \right) - t^{2 D} P^{\CC^*}\left(\mathbb{P}\left(V^+\right)\right) - t^{2 D} P^{\CC^*}\left(\mathbb{P}\left(V^-\right)\right) \\
=& P(B\CC^*) \left( P\left(\mathbb{P}\left(V^+\oplus V^-\right)\right) -2 t^{2 D} P\left( \mathbb{P}(V^+)\right) \right) \\ 
=& \frac{1}{1-t^2}\left( \frac{1-t^{4 D}}{1-t^2}-2t^{2 D}\frac{1-t^{2D}}{1-t^2}\right)
\end{align*}
The intersection cohomology of $\left(V^+\oplus V^-\right)\git\CC^*$ can be computed from the cohomology of $\mathbb{P}\left(V^+\oplus V^-\right)\git\CC^*$ with \cite[2.17]{Kirwan1986}. For $k\ge 2D-1$ the intersection cohomology space $IH^k\left(\left(V^+\oplus V^-\right)\git\CC^*, \CC\right)$ is zero and for $k<2D-1$ it is the primitive part of the cohomology $ H^{k}\left( \mathbb{P}\left(V^+\oplus V^-\right)\git\CC^* ,\CC \right)$. The variety $\mathbb{P}\left(V^+\oplus V^-\right)\git\CC^*$ is projective and smooth so that its cohomology is pure, therefore the intersection cohomology of $\left(V^+\oplus V^-\right)\git\CC^*$ is also pure. Moreover the Poincar\'e polynomial for intersection cohomology of $\left(V^+\oplus V^-\right)\git\CC^*$ can be computed from the previous one
\[
P\left(\left(V^+\oplus V^-\right)\git\CC^*; t\right) = \frac{1-t^{2D}}{1-t^2}.
\]
There is only one Hodge polynomial specializing to this Poincar\'e polynomial.
\end{proof}

\begin{theorem}
For $n \ge 2$, the intersection cohomology of $\mathcal{F} \git O(2,\CC)$ is pure and its Hodge polynomial is
\[
H\left(\mathcal{F} \git O(2,\CC)\right) = \left.(1+u)^{g}(1+v)^g\right|_{even}\frac{1-(uv)^{D}}{1-uv},
\]
with $D = (2n - 1)(g - 1)$ and $\left.(1+u)^{g}(1+v)^g\right|_{even}$ the even degree part of the polynomial $(1+u)^{g}(1+v)^g$.
\end{theorem}
\begin{proof}

The identity component of $O(2,\CC)$ is the multiplicative group $\CC^*$ and 
\[
\pi_0\left(O(2,\CC)\right)\cong \ZZ/2\ZZ.
\]
The intersection cohomology of the quotient $\mathcal{F} \git O(2,\CC)$ is the $\ZZ/2\ZZ$-invariant part of the intersection cohomology of $\mathcal{F}\git\CC^*$. The $\CC^*$-action is defined by \eqref{eq_O2_action}. As $\mathcal{F}$ is the total space of a vector bundle, it is the total space of the normal bundle to the zero section. Therefore the intersection cohomology of the quotient $\mathcal{F} \git \CC^*$ can be computed with methods from Kirwan \cite[2.20]{Kirwan1986},
\begin{equation}\label{eq_f0_gm}
IH^*\left(\mathcal{F} \git \CC^*,\CC\right)\cong IH^*\left(\left(V^+\oplus V^-\right) \git \CC^*,\CC\right)\otimes H^*\left(\Pic^0(X),\CC\right).
\end{equation}
With $V^+=H^0\left(X,M K^n\right)$ and $V^-=H^0\left(X,M\inv K^n\right)$ so that $V^+\oplus V^-$ is the fibre of the vector bundle over some $M\in \Pic^0(X)$. Notice that $\dim V^+ = \dim V^- = D$.

The Hodge polynomial for intersection cohomology of $\mathcal{F} \git \CC^*$ can be computed from (\ref{eq_f0_gm}) and Lemma \ref{lemma_ic_fibre},
\[
H\left(\mathcal{F} \git \CC^*\right) = (1+u)^{g}(1+v)^g\frac{1-(uv)^{D}}{1-uv} .
\]
To obtain the intersection cohomology of $\mathcal{F} \git O\left(2,\CC\right)$ consider the $\ZZ/2\ZZ$ action on $IH^*\left(\widetilde{\mathcal{ F}}_0 \git \CC^*, \CC\right)$. Under the isomorphism \eqref{eq_f0_gm} the action is obtained by the tensor product of an action on $IH^*\left(\left(V^+\oplus V^-\right) \git \CC^* , \CC\right)$ and the action on $ H^*\left(\Pic^0(X), \CC\right)$ induced by inversion. As $IH^i\left( \left(V^+\oplus V^-\right) \git \CC^* , \CC \right)$ has dimension $0$ or $1$ and the $\ZZ/2\ZZ$-action preserves the orientation, it must be trivial. Then the invariant part is $IH^*\left(\mathcal{F} \git \CC^*,\CC\right)^{\ZZ/2\ZZ}\cong IH^*\left( \left(V^+\oplus V^-\right) \git \CC^* , \CC \right) \otimes H^*\left(\Pic^0(X),\CC\right)^{\ZZ/2\ZZ} $. The invariant part of $H^*\left(\Pic^0(X) , \CC\right)$ under inversion is the even degree cohomology. Hence we obtain the Hodge polynomial for intersection cohomology,
\[
H\left(\widetilde{\mathcal{ F}}_0 \git O\left(2,\CC\right); u ,v \right)=\left.(1+u)^{g}(1+v)^g\right|_{even}\frac{1-(uv)^{D}}{1-uv},
\]
with $\left.(1+u)^{g}(1+v)^g\right|_{even}$ the even degree terms in $(1+u)^{g}(1+v)^g$.

\end{proof}
\begin{remark}
    Following Remark \ref{remark_SO_0}, the corresponding component in the moduli space of $K^n$-twisted $\SO_0(1,2)$-Higgs bundles has for Hodge polynomial for intersection cohomology
    \[
    H\left(\widetilde{\mathcal{ F}}_0 \git O\left(2,\CC\right); u ,v \right)=(1+u)^{g}(1+v)^g\frac{1-(uv)^{D}}{1-uv}.
    \]
\end{remark}
Using the Cayley correspondence, the previous result about $K^n$-twisted $\SO(1,2)$-Higgs bundles  gives in fact information about $\SO(n,n+1)$-Higgs bundles. We denote by $\mathcal{M}^{\mathcal{C},0}_K\left(\SO(n,n+1)\right)$ the corresponding Cayley component in $\mathcal{M}_K\left(\SO(n,n+1)\right)$.
\begin{corollary}
The intersection cohomology of the component $\mathcal{M}^{\mathcal{C},0}_K\left(\SO(n,n+1)\right)$ is pure and its Hodge polynomial is
\[
H\left(\mathcal{M}^{\mathcal{C},0}_K\left(\SO(n,n+1)\right); u , v\right) = \left.(1+u)^{g}(1+v)^g\right|_{even}\frac{1-(uv)^{D}}{1-uv}.
\]
\end{corollary}
\begin{remark}
Following Proposition \ref{prop_sopq0}, for $n$ even, previous corollary also describes the Cayley component $\mathcal{M}^{\mathcal{C},0}_K\left(\SO_0(n,n+1)\right)\subset \mathcal{M}_K\left(\SO_0(n,n+1)\right)$
\[
H\left(\mathcal{M}^{\mathcal{C},0}_K\left(\SO_0(n,n+1)\right); u , v\right) = \left.(1+u)^{g}(1+v)^g\right|_{even}\frac{1-(uv)^{D}}{1-uv}.
\]
For $n$ odd, following Proposition \ref{prop_sopq0} and Remark \ref{remark_SO_0}
\[
H\left(\mathcal{M}^{\mathcal{C},0}_K\left(\SO_0(n,n+1)\right); u , v\right) = (1+u)^{g}(1+v)^g\frac{1-(uv)^{D}}{1-uv}.
\]
\end{remark}
\subsection{\texorpdfstring{$\SO_0(n,n+2)$-Higgs bundles, $n$ odd}{SO_0(n,n+2)-Higgs bundles, n odd} }\label{sonn+2}
We first note that the identity component $\SO_0(1,3)\subset\SO(1,3)$ is isomorphic to $\PGL(2,\CC)$. The intersection cohomology of $\mathcal{M}_{K^n}\left(\PGL(2,\CC)\right)$ is known (see Corollary \ref{cor_pgl}). Following Proposition \ref{prop_sopq0}, for $n$ odd, the group $\SO_0(1,3)$ is the Cayley partner of $\SO_0(n,n+2)$. Therefore we can describe the Cayley components $\mathcal{M}_K^{\mathcal{C},d}\left(\SO_0(n,n+2)\right)\subset \mathcal{M}_K\left(\SO_0(n,n+2)\right)$ for $d\in \ZZ/2\ZZ$. They satisfy
\[
\mathcal{M}_K^{\mathcal{C},d}\left(\SO_0(n,n+2)\right) \cong \mathcal{M}^d_{K^n}\left(\PGL(2,\CC)\right) \otimes \bigoplus_{j=1}^{n-1}H^0(X,K^{2j}).
\]

Using the results about intersection cohomology of moduli spaces of Higgs bundles recalled in Section \ref{sect_ic_higgs_bundles}, we can obtain the intersection cohomology of the components $\mathcal{M}_K^{\mathcal{C},d}\left(\SO(n,n+2)\right)$, for $d\in\ZZ/2\ZZ$,
\begin{equation*}
IH^*_c\left(\mathcal{M}_K^{\mathcal{C},d}\left(\SO_0(n,n+2)\right),\CC\right) \cong IH^*_c\left( \mathcal{M}^d_{K^n}\left(\PGL(2,\CC)\right)  ,\CC\right).
\end{equation*}
Therefore the intersection cohomology of those Cayley components are given by Corollary \ref{cor_pgl}, explicitly
\begin{multline}
P_c\left( \mathcal{M}_K^{\mathcal{C},d}\left(\SO_0(n,n+2)\right) \right) = t^{(2g-2)(6n -2)}\left(t^{(2g-2)(2n-2)}\frac{(t^3-1)^{2g}}{(t^4-1)(t^2-1)}\right. \\
\left.+\frac{(t-1)^{2g}}{t^2-1}\left(\frac{1}{4}  +\frac{g}{t-1} - \frac{1}{2(t^2-1)} -(n-1)(g-1) \right) - \frac{1}{4}\frac{(t-1)^{2g}}{t^2+1}\right).
\end{multline}
It is not obvious that this expression is actually a polynomial with integer coefficients. For $g=2$ a computation with the software {\it SageMath} gives
\begin{multline*}
  P_c\left( \mathcal{M}_K^{\mathcal{C},d}\left(\SO_0(3,5)\right) \right) =  3 t^{32} -10 t^{33} + 15 t^{34} -16 t^{35} + 15 t^{36} -12 t^{37} \\ + 9 t^{38} -8 t^{39} + 8 t^{40} -4 t^{41} + 2 t^{42} -4 t^{43} +  t^{44} +  t^{46},
\end{multline*}
and
\begin{multline*}
  P_c\left( \mathcal{M}_K^{\mathcal{C},d}\left(\SO_0(5,7)\right) \right) = 5 t^{56} - 18 t^{57} + 29 t^{58} - 32 t^{59} + 31 t^{60} - 28 t^{61} + 25 t^{62} \\ - 24 t^{63}  + 23 t^{64} - 20 t^{65} + 17 t^{66} - 16 t^{67} + 15 t^{68} - 12 t^{69} + 9 t^{70} - 8 t^{71} \\+ 8 t^{72} - 4 t^{73} + 2 t^{74} - 4 t^{75} +  t^{76} +  t^{78}.
\end{multline*}

\section{Higher rank Teichm\"uller components for quaternionic real form of type \texorpdfstring{$\E_6$}{E_6}}\label{e6}
Let $\E^{2}_6$ be the quaternionic real form of the adjoint form of the complex group  $\E_6$. We consider now the Cayley correspondence for this group.
This is described for type $E_6$ in  \cite[Proposition 4.8]{Bradlow_Collier_Garcia-Prada_Gothen_Oliveira}  at the level of Lie algebra.  The Lie algebra of the group $G'$ appearing in the Cayley correspondence in this situation  is $\slf_3\left(\CC\right)$. In the following theorem we identifying the group $G'$ and give explicitly the particular form of the Cayley correspondence for $\E^{2}_6$.
\begin{theorem}
    The Cayley correspondence for $\E^{2}_6$ is an open and closed embedding
    \[
    \Psi : \mathcal{M}_{K^4}\left(\PGL(3,\CC)\right)\times H^0(X,K^2) \times H^0(X,K^6) \to \mathcal{M}_K\left(E^2_6\right).
    \]
\end{theorem}
\begin{proof}
    We follow the construction from \cite{Bradlow_Collier_Garcia-Prada_Gothen_Oliveira}. The following picture represents both the weighted Dynkin diagram of the relevant magical $\slf_2$-triple $\left\lbrace f,h,e\right\rbrace$ and the Satake diagram of the associated quaternionic real form. The integer label at each vertex  represents the value of the corresponding simple root at $h$
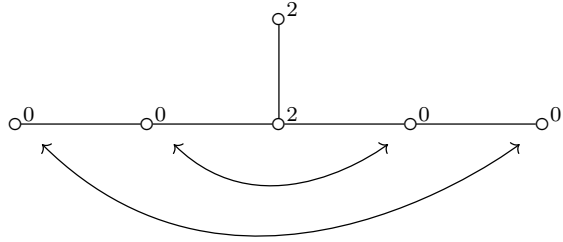
    \[\begin{tikzcd}
	&& \circ^2 \\
	\circ^0 & \circ^0 & \circ^2 & \circ^0 & \circ^0
	\arrow[no head, from=2-1, to=2-2, start anchor={[xshift = -11]}, end anchor = {[xshift = 6]} ]
	\arrow[no head, from=2-3, to=2-4,start anchor={[xshift = -11]}, end anchor = {[xshift = 6]} ]
	\arrow[no head, from=2-4, to=2-5,start anchor={[xshift = -11]}, end anchor = {[xshift = 6]}]
	\arrow[no head, from=1-3, to=2-3,start anchor={[xshift=-2.3,yshift = 5.5]}, end anchor = {[xshift = -2.3,yshift = -9]} ]
	\arrow[no head, from=2-2, to=2-3,start anchor={[xshift = -11]}, end anchor = {[xshift = 6]}]
 \arrow[<->,from = 2-1, to = 2-5, controls={+(3,-3) and +(0,0)}, end anchor = south west]
  \arrow[<->,from = 2-2, to = 2-4, controls={+(1.5,-1.5) and +(0,0)}, end anchor = south west]
\end{tikzcd}.\]
Using standard notations (see Humphreys \cite[26.3]{Humphreys}), the group $\E_6$ is generated by its subgroup $U_{\alpha}$ for $\alpha$ a root of the root system of $\E_6$. As we consider an adjoint form of $\E_6$, the subgroups $G_{\alpha}$ spanned by $U_{\alpha}$ and $U_{-\alpha}$ are isomorphic to $\PSL_2\left(\CC\right)$. 
Let $G_0$ be the stabilizer of $h$ in $\E_6$, it is  generated by the maximal torus $T$ and the $U_{\alpha}$ with $\alpha(h)=0$. The Dynkin diagram of the semisimple part $G_{0,ss} \subset G_0$ is
\[\begin{tikzcd}
	\circ & \circ &  & \circ & \circ
	\arrow[no head, from=1-1, to=1-2, start anchor={[xshift = -6]}, end anchor = {[xshift = 6]} ]
	\arrow[no head, from=1-4, to=1-5,start anchor={[xshift = -6]}, end anchor = {[xshift = 6]}]
\end{tikzcd}\]
and its subgroups $G_{\alpha}$ are isomorphic to $\PSL_2\left(\CC\right)$, and  hence $G_{0,ss}\cong \PSL_3(\CC)\times \PSL_3(\CC)$. The group $G'$ appearing in the Cayley transform is a real form of $G_{0,ss}$, at the level of Lie algebra this real form is $\mathfrak{g}^{\RR}_{0,ss}\cong \slf_3(\CC)$ hence $G'\cong \PSL_3\left(\CC\right)$. The integers appearing in the Cayley correspondence are described in \cite[Lemma 5.7]{Bradlow_Collier_Garcia-Prada_Gothen_Oliveira}, they are $m_c = 3$, $l_1= 1$ and $l_2 = 5$.
\end{proof}

We denote the corresponding Cayley components by $\mathcal{M}^{\mathcal{C},d}_K(\E_6^2)$, then for $d\in \ZZ/3\ZZ$
\[
IH^*_c\left(\mathcal{M}^{\mathcal{C},d}_K(\E_6^2),\CC\right)\cong IH^*_c\left(\mathcal{M}^d_{K^{4}}\left(\PGL(3,\CC)\right),\CC\right).
\]
Therefore the intersection cohomology of these Cayley components is given by Corollary \ref{cor_pgl}.
\begin{multline*}
P_c\left(\mathcal{M}^{\mathcal{C},d}_K(\E_6^2)\right) = \frac{t^{90g-90}(t-1)^{4g}}{6(t^2-1)^2}\left(3(6g-6)^2+\frac{2g(g-1)}{(t-1)^2}+\frac{4g}{t-1}+ \frac{10g^2}{(t-1)^2}\right. \\
\left.+(6g-6)\left(\frac{6}{t^2-1}-3-\frac{12g}{t-1}\right) -\frac{12g}{(t-1)(t^2-1)}+\frac{4}{(t^2-1)^2} +\frac{2}{3} -\frac{2}{t^2-1}\right) \\
-\frac{t^{102g-102}(t-1)^{2g}(t^3-1)^g }{t^2-1}\left(6g-6 -\frac{2g}{t^3-1}+\frac{1}{t^2-1}+\frac{1}{t^4-1}\right) \\
+\frac{t^{126g-126}(t^3-1)^{2g}(t^5-1)^{2g}}{(t^2-1)(t^4-1)^2(t^6-1)}-\frac{t^{90g-90}(1+t+t^2)^{2g}(t^2-1)}{9(t^6-1)}
\end{multline*}
It is not obvious that this expression is actually a polynomial with integer coefficients. It can be computed with {\it SageMath} for small values of $g$, for instance when $g=2$,
\begin{multline*}
   P_c\left(\mathcal{M}^{\mathcal{C}}_K(\E_6^2)\right) =  36 t^{90} -256  t^{91} + 848 t^{92} -1816  t^{93} + 2919 t^{94} -3936  t^{95} + 4776 t^{96} \\
    -5432  t^{97} + 5920 t^{98} -6248  t^{99} + 6424 t^{100} -6464  t^{101} + 6384 t^{102} -6200   
    t^{103}\\ + 5919 t^{104} -5568  t^{105} + 5190 t^{106} -4776  t^{107} + 4330 t^{108} -3912  
    t^{109}\\ + 3526 t^{110} -3136  t^{111} + 2765 t^{112} -2448  t^{113} + 2162 t^{114} -1880 
    t^{115}\\ + 1623 t^{116} -1408  t^{117} + 1216 t^{118} -1032  t^{119} + 865 t^{120} -728  
    t^{121}\\ + 614 t^{122} -504  t^{123} + 403 t^{124} -328  t^{125} + 269 t^{126} -208  
    t^{127}\\ + 158 t^{128} -124  t^{129} + 93 t^{130} -72  t^{131} + 49 t^{132} -32  t^{133} + 29 t^{134}\\ -16  t^{135} + 10 t^{136} -8  t^{137} + 3 t^{138} -4  t^{139} +  t^{140} +  t^{142}.
    \end{multline*}
\section{\texorpdfstring{Cayley components for $\U(n,n)$ and $\PU(n,n)$}{Cayley components for U(n,n) and PU(n,n)}}
Both $\SU(n,n)$ and $\PU(n,n)$ are semisimple and Hermitian of tube type. Even if $\U(n,n)$ is note semisimple it admits a general Cayley correspondence, we recall the explicit description in this section. Therefore there is a general Cayley correspondence for the groups $\SU(n,n)$, $\U(n,n)$ and $\PU(n,n)$. The partners are respectively $\SL_n(\CC)\rtimes \RR^*$, $\GL(n,\CC)$ and
$\PGL(n,\CC)$. In Section \ref{sect_ic_higgs_bundles} we recalled various results about twisted Higgs bundles for the groups $\GL(n,\CC)$ and $\PGL(n,\CC)$. These results allow us to obtain the intersection cohomology of the Cayley components for $\U(n,n)$ and $\PU(n,n)$. First let us recall the Cayley correspondence for this groups.
\begin{definition}
    A $\U(n,n)$-{\em Higgs bundle} is a tuple $(V,W,\beta,\gamma)$ consisting of 
  holomorphic vector bundles $V$ and $W$ of rank $n$, and two Higgs fields $\beta : W \to V \otimes K$ and $\gamma : V \to W \otimes K$.
A $\SU(n,n)$-{\em Higgs bundle} is a $\U(n,n)$-Higgs bundle with the additional condition $\det V = \left(\det W\right)^{-1}$.

A $\PU(n,n)$-{\em Higgs bundle} is an equivalence class of $\U(n,n)$-Higgs bundles under the action of $\Pic(X)$ by tensor product.
\end{definition}

\subsection{\texorpdfstring{$\U(n,n)$-Higgs bundles}{U(n,n)-Higgs bundles}}\label{unn}
For a $\U(n,n)$-Higgs bundle $(V,W,\beta,\gamma)$, let $E:=V\oplus W$ and $\phi : E \to E\otimes K$ the Higgs field induced by $\beta$ and $\gamma$.
\begin{remark}
    A $\U(n,n)$-Higgs bundle is semistable if and only if for all subbundle $E'\subset E$ preserved by $\phi$, $\mu(E')\le \mu(E)$.
\end{remark}
The Toledo invariant of a $\U(n,n)$-Higgs bundle is
\[
\tau := \deg(V) - \deg(W).
\]
The Toledo invariant of a semistable $\U(n,n)$-Higgs bundle satisfies the following generalization of the Milnor--Wood inequality (see \cite{Bradlow_Garcia-Prada_Gothen}) 
\[
|\tau| \le 2 n (g-1).
\]
Consider a semistable $\U(n,n)$-Higgs bundle with maximal Toledo invariant $\tau = 2 n (g-1)$. Semistability implies that $\gamma$ is an isomorphism (see \cite[Lemma 3.24]{Bradlow_Garcia-Prada_Gothen}). This semistable $\U(n,n)$-Higgs bundle with maximal Toledo invariant therefore defines a semistable $K^2$-twisted Higgs bundle with Higgs field $\left(\beta\otimes \Id_K\right) \circ \gamma : V \to V \otimes K^2$. Reciprocally a semistable $K^2$-twisted Higgs bundle $V$ with Higgs field $\phi : V \to V\otimes K^2$ defines a semistable $\U(n,n)$-Higgs bundle. This described exactly the Cayley map, we add an index $d$ for the degree of $V$ so that the Cayley components are denoted by $\mathcal{M}^{\mathcal{C},d}_K\left(\U(n,n)\right)$, 
\[
\Psi : \mathcal{M}^d_{K^2}\left(\GL(n,\CC)\right) \xrightarrow{\sim} \mathcal{M}^{\mathcal{C},d}_K\left(\U(n,n)\right).
\]
The Image of these Cayley maps are exactly the components of maximal Toledo invariant in the moduli space of polystable $\U(n,n)$-Higgs bundles and we have the following theorem. 
\begin{theorem}
    The intersection cohomology of the connected components of maximal Toledo invariants in the moduli space of polystable $\U(n,n)$-Higgs bundle satisfies
    \[
    IH^*_c\left(\mathcal{M}^{\mathcal{C},d}_K\left(\U(n,n)\right),\CC\right) \cong IH^*_c\left(\mathcal{M}^d_{K^2}\left(\GL(n,\CC)\right),\CC\right)
    \]
    and is given by setting $l=2$ in Theorem \ref{th_Mozgovoy_OGorman}.
\end{theorem}
\begin{remark}
    Mozgovoy and Schiffmann, in their article about counting twisted Higgs bundles \cite{Mozgovoy_Schiffmann}, also studied invariants for $\U(n,n)$--Higgs bundles. They give a close formula for the volume of the stack of $\U(n,n)$-Higgs bundle. Solving the Harder--Narasimhan recursion would then provide a formula for the volume of the stack of semistable $\U(n,n)$--Higgs bundles. The previous theorem also relies on \cite{Mozgovoy_Schiffmann} but provides another invariant.
\end{remark}
\subsection{\texorpdfstring{$\PU(n,n)$-Higgs bundles}{{PU(n,n)}-Higgs bundles}}\label{punn}
Considering equivalence class of Higgs bundles up to tensor product by a line bundle of degree $0$, the previous description induces the Cayley correspondence for $\PU(n,n)$. The index $d$ now lies in $\mathbb{Z}/n\mathbb{Z}$
\[
\Psi : \mathcal{M}^d_{K^2}\left(\PGL(n,\CC)\right) \to \mathcal{M}^{\mathcal{C},d}_K\left(\PU(n,n)\right).
\]
\begin{sloppypar} The general theory of the Cayley correspondence for simple groups apply to $\PU(n,n)$. Therefore, under the
  non-abelian Hodge correspondence, the Cayley component $\mathcal{M}^{\mathcal{C},d}_K\left(\PU(n,n)\right)$ corresponds to a Higher rank Teichm\"uller component $\mathcal{R}^{\mathcal{T},d}\left(\PU(n,n)\right)\subset \mathcal{R}\left(\PU(n,n)\right)$. 
\end{sloppypar}

\begin{theorem}
    The Poincar\'e polynomial for compactly supported intersection cohomology of the higher rank Teichm\"uller component $\mathcal{R}^{\mathcal{T},d}\left(\PU(n,n)\right)$ is
    \[
    P_c\left(\mathcal{R}^{\mathcal{T},d}\left(\PU(n,n)\right)\right) = P_c\left(\mathcal{M}_{K^2}^{d}\left(\PGL(n,\CC)\right)\right)
    \]
    which is given by setting $l=2$ in Corollary \ref{cor_pgl}.
\end{theorem}

\subsection{\texorpdfstring{$\U(n,n)$}{U(n,n)}-character variety}
There also exists a non-abelian Hodge correspondence for the group $\U(n,n)$, the character variety is
\[
    \widetilde{\mathcal{R}}(\U(n,n)) := \Hom^{\mathcal{Z},\rd}\left(\Gamma,\U(n,n)\right) / \U(n,n).
\]
with $\Gamma$ a central extension of $\pi_1(\Sigma_g)$, then from Theorem \ref{th_gamma_nah} we have that the moduli space $\mathcal{M}_K(\U(n,n))$ is homeomorphic to $\widetilde{\mathcal{R}}(\U(n,n))$.  Notice that representations in $\widetilde{\mathcal{R}}(\U(n,n))$ are not faithful as the image of the central element in $\Gamma$ is of finite order. Thus we might not want to directly generalize the definition of higher rank Teichm\"uller components to real reductive groups. However by comparing with $\PU(n,n)$, the representations in the components of
$\widetilde{\mathcal{R}}(\U(n,n))$ corresponding to Cayley components induce faithful and discrete representations of $\pi_1(\Sigma_g)$ after composition with $\U(n,n)\to \PU(n,n)$.

\bigskip

{\em Acknowledgements}. We wish to thank Brian Collier for many interesting discussions and suggestions. The second author thanks the IHES for its hospitality and support during the final stages of the preparation of this paper.


\end{document}